\documentclass[a4paper]{article}
\usepackage{amsmath, amssymb}
\usepackage{stmaryrd}
\usepackage{amsthm,cite}
\usepackage{amsfonts}
\usepackage[mathscr]{eucal}
\usepackage{graphicx,subfigure}
\usepackage{epsfig,color,pifont,enumerate}
\usepackage{setspace}
\graphicspath{{img/}}

\usepackage{geometry}
\newcommand{\real}{\mathbb{R}}

\renewcommand{\phi}{\varphi}
\renewcommand{\epsilon}{\varepsilon}
\renewcommand{\kappa}{\varkappa}
\newcommand{\sgn}{\text{\bf sgn}\,}
\newcommand{\trs}{\top}
\newcommand{\br}{\mathbb{R}}
 \newcommand{\bldr}{\boldsymbol{r}}
 \newcommand{\bldv}{\boldsymbol{v}}
\newcommand{\blde}{\boldsymbol{e}}
\newcommand{\bldt}{\pmb{T}}
\newcommand{\bldn}{\pmb{N}}

\newtheorem{theorem}{Theorem}[section]

\newtheorem{remark}{Remark}[section]
\newtheorem{assumption}{Assumption}[section]
\newtheorem{lemma}{Lemma}[section]

\newtheorem{proposition}{Proposition}[section]

\newcommand{\so}{\text{\scriptsize $\mathcal{O}$}}
\newcommand{\spr}[2]{\left\langle #1; #2 \right\rangle}

\newcommand{\ov}[1]{\overline{#1}}
\newcommand{\blr}{r}

\newcommand{\dd}[1]{#1^{\prime\prime}}

\newcommand{\epf}{$\qquad \Box$}

\newcommand{\pp}{{\prime\prime}}
\title{Proofs of Technical Results Justifying and Illustrating an Algorithm of Navigation for Monitoring Unsteady Environmental Boundaries}
\date{}
\author{Alexey S. Matveev, Michael C. Hoy, Kirill S. Ovchinnikov, \\ Alexander M. Anisimov, and Andrey V.Savkin}
\begin{document}
\maketitle
\section{Introduction}
This paper deals with a problem of timely detection and monitoring of re-shaping environmental boundaries.
 This task is of interest for tracking of oil or chemical spills \cite{ClaFie07}, exploration of radioactively contaminated areas, and for many other missions, such as tracking forest fires \cite{CKBMLM06},
or contaminant clouds \cite{WTASZ05}, or harmful algae blooms \cite{PeDuJoPo12}, exploration of sea temperature
and salinity or hazardous weather conditions like tropical storms, etc.
 In fact, such missions are devoted to localization of a spatially distributed phenomenon whose boundary can typically be defined as the level set (isoline) where a certain scalar field (e.g., the radiation level or the concentration of a pollutant) assumes a critical value.
In many such missions, sensors observe the phenomenon only in a point-wise fashion at their locations.
\par
Networks of static sensors typically require high both deployment density and computational/communication loads to provide a good accuracy of observation \cite{NoMi03}.
More effective use of sensors is achieved in mobile networks, where each sensor can explore many locations and
the sensors may be driven to the explored boundary, thus concentrating the overall sensing capacity of the network on the structure of the main interest. 
\par
Autonomous unmanned mobile robots and robotic networks
have been greatly used in the recent past for a variety of such missions in hazardous and complex environments not only to avoid or mitigate risks to human life and health but also due to their
lightweights, inexpensive components, and low power consumptions, see
e.g. \cite{QBGG05,AKP06,GHH04,WSS08,KHBMC05,ClaFi05,WTASZ05,CKBMLM06,KeBeMa04,HJMNTBM05,BertKeMar04,ClaFie07,JAHB09} and references therein.
These robots are often subjected to limitations on communication and so should be equipped with navigation and control systems that enable them to autonomously operate either constantly or for extended periods of time and distance.
\par
Recently, designs of such systems have gained much interest in control community.
Many works in this area assume access to the field gradient or even higher derivatives; see, e.g., \cite{MaBe03,SrRaKu08,ZhLe10,HsLoKu07} and literature therein. For example, they present gradient-based networked contour estimation \cite{MaBe03,BertKeMar04,SrRaKu08}, centralized control laws stemming from the 'snake' algorithms in image segmentation \cite{MaBe03,BertKeMar04},
potential-based approach \cite{HsLoKu07}, estimation of the gradient and Hessian of a noisy field for driving the center of a rigid formation of sensors along a level curve \cite{ZhLe10}.
However, spatial derivatives are often not directly measured.
Meanwhile gradient estimation needs tight-flocking concentration of the sensors near the examined point, contrary to efficient utilization of the network, which ideally means its uniform distribution over the explored boundary. Data exchange needed for this estimation
may be critically degraded by limitations on communication. Finally, sensor-based derivative estimators are prone to noise amplification, and their implementation is an intricate problem in practical setting \cite{AhAb07,chartrand11,VaKh08}.
This carries a threat of performance degradation, puts strong extra burden on controller tuning, and may drastically increase the overall computational load \cite{AhAb07}.
\par
A single mobile sensor with a point-wise access to the field value only is the main target for gradient-free approaches; see e.g., \cite{BaRe03,KeBeMa04,CKBMLM06,Anders07} and literature therein.
Switches between two steering angles, which are carried out whenever the field reading crosses the desired value,
were advocated in \cite{ZhBer07,JAHB09}; a similar approach with more alternative angles was applied to an underwater vehicle
in \cite{BaRe03}. Such methods result in jagged trajectories
and rely in effect on systematic sideways maneuvers to collect enough data, which gives rise to concerns about waste of resources.
A method to control an aerial vehicle based on segmentation of the infrared images of the forest fire was presented in \cite{CKBMLM06}.
These works are based, more or less, on heuristics and provide no rigorous justification of the control laws and guarantees of success.
In \cite{BarBail07}, a linear PD controller fed by the field value was proposed for steering a Dubins-car like vehicle with unlimited control range along a level curve of a radial harmonic field, and was supplied with a local convergence result.
A sliding mode controller
for tracking environmental level sets without gradient estimation was offered in
\cite{MaTeSa12}.
\par
On the whole, previous research was confined to only steady fields for which the environmental boundary does not change over time.
On the contrary, such boundaries are almost never steady in the real world, due to dispersion, advection, transport by wind, drift by water current, etc.
Meanwhile the theory of tracking dynamic level sets lies in the uncharted territory. 
As a particular case, this topic includes navigation of a mobile robot towards an unknowingly maneuvering target and further escorting it with a desired margin\footnote{Which should be given by the desired field value.} on the basis of a single measurement that decays away from the target, like the strength of the infrared, acoustic, or electromagnetic signal, or minus the distance to the target.
Such navigation is of interest in many areas \cite{ADB04,GS04,MTS11ronly}; it carries a potential to reduce the hardware complexity and cost and to improve target pursuit reliability.
To the best of our knowledge, rigorous analysis of such a navigation law was offered in \cite{MTS11ronly} for only a very special case of the field --- the distance to a moving Dubins-like car. However the results of \cite{MTS11ronly} are not applicable to more general dynamic fields.
\par
This paper is aimed at filling the above gaps concerned with dynamic fields, while disembarrassing boundary tracking control from the intricacies related to gradient estimation and sideways fluctuations.
To this end, we consider an under-actuated non-holonomic Dubins-car type mobile robot. It travels with a constant speed over planar curves of bounded curvatures and is controlled by the upper limited angular velocity.
The tracked boundary is given by a generic dynamic scalar field. The robot has access to only the field value at the current location and the rate at which this reading evolves over time via, e.g., numerical differentiation.
\par
The proposed navigation law is non-demanding with respect to computation and motion. It develops some ideas set forth in \cite{MaTeSa12}; e.g.,  gradient estimates and systematic exploration maneuvers are not employed. Whereas the results of \cite{MaTeSa12} are not applicable to unsteady fields, the objective of this paper is to show that the ideas from \cite{MaTeSa12} basically remain viable for generic dynamic fields.
\par
The extended introduction and discussion of the proposed control law are given in the paper submitted by the authors to the IFAC journal Automatica. 
This text basically contains the proofs of the technical facts underlying justification of the convergence and performance of the proposed algorithm in that paper, as well as illustrations of the main theoretical results, which were not included into that paper due to the length limitations. To make the current text logically consistent, we reproduce the problem statement and notations.

\section{Problem Setup and the Control Algorithm}\label{sec1}
\label{sec1}
A mobile robot travels in a plane with a constant speed $v$ and is controlled by the time-varying angular velocity $u$ limited by a given constant $\ov{u}$. There is an unknown and dynamic scalar field $D(t,\boldsymbol{r})$, i.e., a quantity $D \in \br$ that depends on time $t$ and point $\bldr \in \br^2$ in the plane.
Here $\boldsymbol{r}:= (x,y)^\trs$ stands for the pair of the absolute Cartesian
coordinates $x,y$ in $\br^2$.
It is required to drive the robot to the spatial isoline $D(t,\bldr) = d_0$ where the
field assumes a given value $d_0$ and to subsequently drive the robot in a close proximity of this isoline so that the robot circulates along it. The robot has access to the field
value $d(t):= D(t,x,y)$ at the robot's current location $x=x(t), y =y(t)$ and to the rate $\dot{d}(t)$ at which this measurement
evolves over time $t$, via, e.g., numerical differentiation of $d(t)$.
However, neither the partial derivative $D^\prime_t$, nor $D^\prime_x$, nor $D^\prime_y$ is accessible.
\par
The kinematic model of the robot is as follows:
\begin{equation}
\label{1}
\begin{array}{l}
\dot{x} = v \cos \theta,
\\
\dot{y} = v \sin \theta,
\end{array}
\quad
\begin{array}{l}
\dot{\theta} = u , \\ |u|\leq \overline{u},
\end{array}
\quad
\begin{array}{l}
x(0) = x_{\text{in}}
\\
y(0) = y_{\text{in}}
\end{array}, \; \theta(0) = \theta_{\text{in}},
\end{equation}
where $\theta$ is the orientation angle of the robot. 
These equations are classically used
to describe, for example, wheeled robots and missiles (see e.g.
\cite{ManS04,ManS06,TS209} and the literature therein).
\par
It is required to design a controller such that $D[t,x(t),y(t)]
\to  d_0$ as $t \to \infty$.
\par
In this paper, we examine the following navigation law:
\begin{equation}
\label{c.a}
u(t)=-\sgn\{\dot{d}(t)+\chi[d(t)-d_{0} ] \} \bar{u},
\end{equation}
where $\sgn a$ is the sign of $a$ ($\sgn 0:= 0$)
and $\chi(\cdot)$ is a linear function with saturation:
\begin{equation}
\label{chi}
 \chi(p):=
\begin{cases}
\gamma p & \text{if}\; |p|\leq \delta \\
\sgn(p)\mu & \text{otherwise}
\end{cases}, \qquad \mu :=\gamma \delta .
\end{equation}
Here $\gamma >0$ and $\delta >0$ are design parameters.
\par
For the discontinuous controller \eqref{c.a}, the desired
dynamics \cite{UT92} is given by $\dot{d}(t) = - \chi[d(t)-d_0]$.
Since it is unrealistic to desire large $\dot{d}$,
saturation is a reasonable option.
\section{Quantities characterizing dynamic fields}
\label{sec.quant}
\setcounter{equation}{0}
To judge the feasibility of the control objective and to tune the controller, we need some characteristics of the unsteady field $D(\cdot)$ or their estimates. Now we introduce the relevant quantities, along with employed notations.
\begin{itemize}
\item $\bldr$ --- a point on the plane;
\item $\spr{\cdot}{\cdot}$ --- the standard inner product in the plane;
\item $\mathscr{R}_\beta = \left( \begin{smallmatrix} \cos \beta & - \sin \beta \\ \sin \beta & \cos \beta\end{smallmatrix}\right)$ --- the matrix of counter-clockwise rotation through angle $\beta$;
\item $\nabla = \left( \frac{\partial}{\partial x} , \frac{\partial}{\partial y}  \right)^\trs$ --- the spatial gradient;
    \item $D^{\prime\prime}$ --- the spatial Hessian;
\item $I(t,d_\ast) := \{\bldr: D(t,\bldr) = d_\ast \}$ --- the spatial isoline;
\item $[\bldt,\bldn]= [\bldt(t,\bldr),\bldn(t,\bldr)]$ --- the Frenet frame of the spatial isoline $I[t,d_\ast]$ with $d_\ast:= D(t,\bldr)$, i.e.,
    $
\bldn(t,\bldr) = \frac{\nabla D(t,\bldr)}{\|\nabla D (t,\bldr)\|}
$
and the unit tangent vector $\bldt(t,\bldr)$ is oriented so that when traveling on $I[t,d_\ast]$ one has the domain of greater values $G_t^{d_\ast}:=\{ \bldr^\prime : D(t,\bldr^\prime) > d_\ast \}$ to the left;
\item $\varkappa$ --- the signed curvature of the spatial isoline;\footnote{This is positive and negative on the convexities and concavities, respectively, of the boundary of $G^{d_\ast}_t$.}
\item $\bldr_+^{d_\ast(\cdot)}(\Delta t| t,\bldr)$ --- the nearest (to $\bldr$) point of intersection  between the ordinate axis of the Frenet frame $[\bldt(t,\bldr),\bldn(t,\bldr)]$ and the displaced isoline $I[t+\Delta t,d_\ast(t+\Delta t)]$, where the given smooth function $d_\ast(\cdot)$ is such that $d_\ast(t)= D(t,\bldr)$; see Fig.~\ref{mov.isl};
    \item $p^{d_\ast(\cdot)}(\Delta t| t,\bldr)$ --- the ordinate of $\bldr_+^{d_\ast(\cdot)}(\Delta t| t,\bldr)$;
    \item $\lambda^{d_\ast(\cdot)}(t,\bldr)$ --- the front velocity of the isoline:
    $
\lambda^{d_\ast(\cdot)}(t,\bldr):=\lim_{\Delta t\to 0}\frac{p^{d_\ast(\cdot)}(\Delta t| t,\bldr)}{\Delta t };
$
if $d_\ast(\cdot) \equiv \text{const}$, the index $^{d_\ast(\cdot)}$ is dropped in the last three notations;
\item $\alpha(t,\bldr)$ --- the front acceleration of the spatial isoline $I[t,d_\ast], d_\ast:= D(t,\bldr)$ at time $t$ at the point $\bldr$:
    \begin{equation}
\label{alpha_def}
\alpha(t,r):= \lim_{\Delta t\to 0}\frac{\lambda[t+\Delta t, \bldr_+(\Delta t)] - \lambda[t,\bldr]}{\Delta t },
\end{equation}
where $\bldr_+(\Delta t) := \bldr_+(\Delta t| t,\bldr)$;
\item $\Delta \varphi(\Delta t|t,\bldr)$ --- the angular displacement of $T[t+\Delta t, \bldr_+(\Delta t) ]$ with respect to $T[t,\bldr]$; see Fig.~\ref{mov.isl};
\item $\omega(t,\bldr)$ --- the angular velocity of rotation of the isoline $I[t,d_\ast], d_\ast:= D(t,\bldr)$, i.e.,
    $
\omega(t,\bldr):=\lim_{\Delta t\to 0}\frac{\Delta \varphi(\Delta t)}{\Delta t };
$
\item $q(\Delta d| t,\bldr)$ --- the ordinate of the nearest (to $\bldr$) intersection between the $\bldn$-axis of the Frenet frame $[\bldt(t,\bldr),\bldn(t,\bldr)]$ and the isoline $I(t| t,d_\ast+\Delta d)$;
\item $\rho(t,\bldr)$ --- the density of isolines at time $t$ at point $\bldr$:
$$
\rho(t,\bldr):=\lim_{\Delta d \to 0}\frac{\Delta d}{q(\Delta d| t,\bldr)}; \footnote{It assesses the ``number'' of isolines within the unit distance from $I(t,d_\ast)$, where the ``number'' is evaluated by the discrepancy in the values of $D(\cdot)$ observed on these isolines.}
$$
\item $ v_\rho(t,\bldr)$ --- the evolutional (proportional) growth rate of the above density at time $t$ at point $\bldr$:
\begin{equation}
\label{vrho.def}
v_\rho(t,\bldr):= \lim_{\Delta t\to 0}\frac{\rho[t+\Delta t, \bldr_+(\Delta t)]- \rho(t,\bldr)}{\rho(t,\bldr) \Delta t };
\end{equation}
\item $ \tau_\rho(t,\bldr)$ --- the tangential (proportional) growth rate of the isolines density at time $t$ at point $\bldr$:
\begin{equation}
\label{beta_def}
 \tau_\rho (t,r):= \lim_{\Delta s \to 0}\frac{\rho(t, r+ \bldt \Delta s ) -  \rho(t,r)}{\rho(t,\bldr) \Delta s} ;
\end{equation}
\item $ n_\rho(t,\bldr)$ --- the normal (proportional) growth rate of the isolines density at time $t$ at point $\bldr$:
\begin{equation}
\label{njrm_def}
n_\rho (t,r):= \lim_{\Delta s \to 0}\frac{ \rho(t, r+ \bldn \Delta s ) -   \rho(t,r)}{\rho(t,\bldr) \Delta s} .
\end{equation}
\end{itemize}
\par
\begin{figure}
\centering
\scalebox{0.55}{\includegraphics{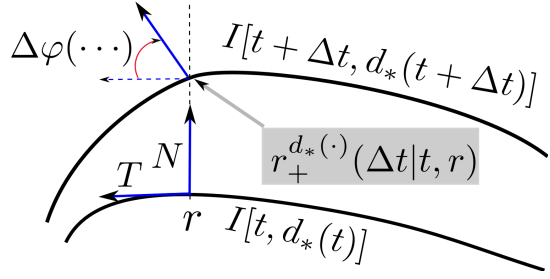}}
\caption{Two close isolines.}
\label{mov.isl}
\end{figure}
\par
Now we show that under minor technical assumptions, these quantities are well-defined, and offer
relationships that explicitly link them to $D(\cdot)$.
\begin{lemma}
\label{lem.relation}
 If $D(\cdot,\cdot)$ is twice continuously differentiable in a vicinity of $(t,\bldr)$ and $\nabla D(t,\bldr) \neq 0$, the quantities from Sec.~{\rm \ref{sec.quant}} are well-defined and the following relations hold at $(t,\bldr)$
 \begin{gather}
 \label{speed1}
 \lambda=-\frac{D'_{t}}{\|\nabla D \|}, \quad  \lambda^{d_\ast(\cdot)} = \lambda + \frac{\dot{d}_\ast}{\|\nabla D \|}, \quad \rho = \|\nabla D\|,
 \\  \bldr_+(dt):= \bldr_+(dt|t,\bldr) = \bldr + \lambda \bldn dt + \so(dt),
  \label{speed1a}
 \\
 \label{alpha}
 v_\rho = \frac{\spr{\nabla D^\prime_{t}+\lambda \dd{D}\bldn}{\bldn}}{\|\nabla D \|},
 \omega= - \frac{\spr{\nabla D^\prime_{t}+\lambda \dd{D}\bldn}{\bldt}}{\|\nabla D\|}, \\
 \label{alpha.true}
\alpha=- \frac{D^{\prime\prime}_{tt} + \lambda \left\langle \nabla D^\prime_t ; \bldn \right\rangle}{\|\nabla D\|}  - \lambda  v_\rho ,
\\
 \label{tau}
\varkappa = - \frac{\spr{\dd{D}\bldt}{\bldt}}{\|\nabla D\|},  \tau_\rho =\frac{\spr{D^{\prime\prime}\bldn}{\bldt}}{\left\| \nabla D \right\|} ,
 n_\rho =\frac{\spr{D^{\prime\prime}\bldn}{\bldn}}{\left\| \nabla D \right\|} .
\end{gather}
\end{lemma}
\begin{proof}
The first claim and \eqref{speed1}, \eqref{speed1a} follow from the implicit function theorem \cite{KrPa02}. With \eqref{speed1}, \eqref{speed1a} in mind,
\begin{gather}
\nonumber
\nabla D\left[t+dt,\bldr_+(dt)\right]=
\nabla D[t+dt,\bldr +\lambda \bldn dt+\so(dt)]
= \nabla D+[\nabla D^\prime_{t}+\lambda \dd{D}\bldn]dt +\so(dt),
\\
\nonumber
\rho(t+\Delta t, \bldr_+(dt)) \overset{\text{\eqref{speed1}}}{=\!=}
\left\| \nabla D\left[t+dt,\bldr_+(dt)\right] \right\|
= \left\| \nabla D+[\nabla D^\prime_{t}+ \lambda \dd{D} \bldn]dt +\so(dt) \right\|
\\
= \left\| \nabla D\right\| + \frac{\left\langle \nabla D; \nabla D^\prime_{t}+\lambda\dd{D} \bldn \right\rangle}{\left\| \nabla D\right\|} dt + \so(dt)
 =
\left\| \nabla D\right\| + \left\langle \bldn; \nabla D^\prime_{t}+ \lambda \dd{D} \bldn \right\rangle dt + \so(dt),
\label{expan}
\end{gather}
which gives the first formula in \eqref{alpha}. Furthermore,
\begin{multline*}
\bldn\left[t+dt,\bldr_+(dt)\right] = \frac{\nabla D\left[t+dt,\bldr_+(dt)\right]}{\left\| \nabla D\left[t+dt,\bldr_+(dt)\right] \right\|}
= \bldn + \so(dt) + \left[ \frac{\nabla D^\prime_{t}+ \lambda \dd{D} \bldn}{\left\| \nabla D\right\|} - \frac{\nabla D}{\left\| \nabla D\right\|^3} \left\langle \nabla D; \nabla D^\prime_{t}+ \lambda  \dd{D}\bldn \right\rangle  \right] dt
\\
=
 \bldn + \frac{1}{\left\| \nabla D\right\|}\left[\nabla D^\prime_{t}+ \lambda \dd{D} \bldn - \bldn \left\langle \bldn; \nabla D^\prime_{t}+ \lambda  \dd{D}\bldn \right\rangle  \right] dt
+ \so(dt)
\overset{\text{(a)}}{=} \bldn + \frac{\left\langle \bldt; \nabla D^\prime_{t}+ \lambda  \dd{D}\bldn \right\rangle} {\left\| \nabla D\right\|} \bldt dt   + \so(dt),
\end{multline*}
 where (a) holds since $\pmb{W} = \langle \pmb{W},\bldt \rangle \bldt + \langle \pmb{W},\bldn \rangle \bldn \; \forall \pmb{W} \in \real^2$. In the Frenet frame $(\bldt,\bldn)$, we have
\begin{equation*}
\bldn \left[t+dt,\bldr_+(dt)\right] = \left( \begin{array}{c}
-\sin \Delta \varphi(dt|t,\bldr)
\\
\cos \Delta \varphi(dt|t,\bldr)
\end{array}\right)
= \bldn + \left( \begin{array}{c}
-\cos 0
\\
- \sin 0
\end{array}\right) \omega dt + \so(dt)
 =
\bldn - \omega \bldt  dt + \so(dt).
\end{equation*}
By equating the coefficients prefacing $\bldt dt$ in the last two expressions, we get the second formula in \eqref{alpha}. Furthermore,
\begin{multline*}
\lambda(t+dt, \bldr_+(dt)) \overset{\text{\eqref{speed1}}}{=}
-\frac{D^\prime_{t}[t+dt, \bldr_+(dt)]}{\|\nabla D(t+dt, \bldr_+(dt))\|}
\overset{\text{\eqref{expan}}}{=}
\\
\lambda - \frac{D^{\prime\prime}_{tt} dt + \left\langle \nabla D^\prime_t ; \bldr_+(dt)-r \right\rangle}{\|\nabla D\|} + D^\prime_t \frac{\left\langle \bldn; \nabla D'_{t}+\lambda D^{\prime\prime} \bldn \right\rangle}{\|\nabla D\|^2} dt +\so(dt)
\\
\overset{\text{\eqref{speed1}}}{=} \lambda - \frac{D^{\prime\prime}_{tt} + \lambda \left\langle \nabla D^\prime_t ; \bldn \right\rangle}{\|\nabla D\|} dt - \lambda  \frac{\left\langle \bldn; \nabla D'_{t}+\lambda D^{\prime\prime} \bldn \right\rangle}{\|\nabla D\|} dt +\so(dt)
\\
\overset{\text{(b)}}{=}
\lambda - \frac{D^{\prime\prime}_{tt} + \lambda \left\langle \nabla D^\prime_t ; \bldn \right\rangle}{\|\nabla D\|} dt - \lambda  v_\rho dt +\so(dt)
\Rightarrow \text{\eqref{alpha.true}},
\end{multline*}
where (b) follows from the first formula in \eqref{alpha}.
The first equation in \eqref{tau} is well known, the second one follows from the transformation
\begin{equation*}
\rho(t, \bldr+ \bldt  ds ) = \left\| \nabla D[t, \bldr + \bldt ds] \right\| = \rho(t,\bldr)  + \frac{\spr{D^{\prime\prime}\bldt}{\nabla D}}{\left\| \nabla D \right\|} ds
+ \so(ds)
= \rho(t,\bldr) + \spr{D^{\prime\prime}\bldt}{\bldn} ds + \so(ds),
\end{equation*}
the third equation in \eqref{tau} is established likewise.
\end{proof}

\section{Assumptions and the Main Theoretical Result}\label{subsec.ass}
\setcounter{equation}{0}
To simplify the matters, we suppose that the operational zone of the robot is 
characterized by the extreme values $d_- \leq d_+$ taken by the field in it
\begin{equation}
\label{m}
\mathcal{M} := \{(t,\bldr): d_- \leq D(t,\bldr) \leq d_+\}
\end{equation}
and contains the required isoline $d_- \leq d_0 \leq d_+$. 
\par
An extended discussion and justification of the following assumptions and results are presented in the main text supported by this paper. 
\begin{proposition}
\label{lem.23}
Suppose that the moving robot remains on the isoline
$D[t,r(t)] \equiv d_0$, in a vicinity of which, the field is twice continuously differentiable and $\nabla D(\cdot,\cdot) \neq 0$. Then at any time the front speed of the isoline at the robot's location does not exceed the speed $v$ of the robot
\begin{equation}
\label{speed}
\left| \lambda[t,\bldr(t)] \right| \leq v,
\end{equation}
\mbox{the robot's velocity $\bldv = v \blde,  \blde=\left( \cos \theta, \sin \theta \right)^{\top}$
has the form}
\begin{equation}
\label{vec.c}
\bldv = \lambda \bldn \pm v_T  \bldt , \; \text{\rm where} \; v_T:= \sqrt{v^2-\lambda^2} ,
\end{equation}
and the following inequality is true
\begin{equation}
\label{accel}
\left| \pm 2 \omega  + \alpha v_T^{-1} + \kappa v_T\right| \leq \overline{u}.
\end{equation}
In $\pm$, the sign $+$ is taken if the robot travels along the isoline so that the domain $\{\bldr : D(t,\bldr) > d_0\}$ is to the left, and $-$ is taken otherwise.
\end{proposition} 
\begin{assumption}
\label{ass.upr}
The field $D(\cdot,\cdot)$ is twice continuously differentiable in the domain \eqref{m} and there exist constants $\Delta_\lambda >0$ and $\Delta_u >0$ such that the following inequalities hold everywhere in $\mathcal{M}$:
\begin{gather}
\label{deltalambda1}
\left| \lambda \right| \leq v - \Delta_\lambda, \\
\label{deltalambda}
\left| \pm 2 \omega  + \frac{\alpha}{\sqrt{v^2-\lambda^2}} + \kappa \sqrt{v^2-\lambda^2}\right| \leq \overline{u} - \Delta_u ,
\end{gather}
where \eqref{deltalambda} is true with the both signs in $\pm$.
\end{assumption}
\par
Since the robot's path is smooth, it is reasonable to exclude isoline singularities. A common guarantee of their nonoccurrence is that the field has no critical points \cite{Thorpe79}. 
\begin{assumption}
\label{ass.grad}
At any time, the zone \eqref{m} does not contain critical points $\bldr$ of the field, i.e., such that $\nabla D =0$. Moreover,
this property does not degrade: there is $b_\rho >0$ such that $\rho(t,\bldr)=\|\nabla D (t,\bldr)\| \geq b_\rho^{-1} \; \forall (t,\bldr) \in \mathcal{M}$.
\end{assumption}
The next assumption is typically fulfilled in the real world, where physical quantities take bounded values.
\begin{assumption}
\label{ass.last}
There exist $b_\lambda, b_\tau, b_n, b_\varkappa, b_v, b_\alpha \in \br$ such that the following inequalities hold in the set \eqref{m}:
\begin{equation}
\label{estim}
|\lambda| \leq b_\lambda, \quad |\tau_\rho| \leq b_\tau, \quad |n_\rho| \leq b_n,
\quad |\varkappa| \leq b_\varkappa, \quad |v_\rho| \leq b_v, \quad |\alpha| \leq b_\alpha.
\end{equation}
\end{assumption}
 Under the control law \eqref{c.a}, the robot moves with $u\equiv \pm \ov{u}$ during an initial time interval. This motion is over the {\em initial circle} $C_\pm^{\text{in}}$, which is the respective path from the initial state given in \eqref{1}.
The last assumption requires that the encircled closed discs $D_\pm^{\text{in}}$ (also called {\it initial}) lie in the operational zone \eqref{m}, and the maximal turning rate of the robot exceeds the average angular velocity of the field gradient, at least on some initial time interval.
\begin{assumption}
\label{ass.disk}
  There exists a natural $k$ such that
  during the time interval $\left[ 0,T_k\right], T_k:=\frac{2 \pi k}{\ov{u}} $ (a) the gradient $\nabla D(t,\bldr_{\text{in}}), \bldr_{\text{in}}:= (x_{\text{in}}, y_{\text{in}})^\trs$ rotates through an angle that does not exceed $2 \pi (k-1)$ and (b) the both initial discs lie in the domain \eqref{m}, i.e., $[0,T_k] \times D_\pm \subset \mathcal{M}$.
\end{assumption}
\begin{theorem}
\label{th.main}
 Let Assumptions~{\rm \ref{ass.upr}}---{\rm \ref{ass.disk}} hold and
 the robot be driven by the navigation law \eqref{c.a} whose
 parameters $\gamma$ and $\mu=\gamma\delta$ satisfy the following inequalities, where $\mu_\ast:= b_\rho \mu$ and $\sigma(\mu_\ast) := \sqrt{\Delta_\lambda^2 - 2 v \mu_\ast - \mu^2_\ast}$:
\begin{equation}
\label{limit}
0< \mu_\ast  <  \sqrt{v^2+\Delta_\lambda^2}-v, \qquad \Delta_u
> \left(3 b_\tau + \frac{b_\varkappa+2b_v + \gamma +b_n \mu_\ast}{\sigma(\mu_\ast)} + \frac{b_\alpha}{\sigma(\mu_\ast)^3}\right) \mu_\ast  .
\end{equation}
Then the robot achieves the control objective $d(t)\xrightarrow{t \to \infty} d_{0}$, always remaining in the domain \eqref{m}.
\end{theorem}
\par
By sacrificing conservatism, the requirements \eqref{limit} to the controller parameters can be
transformed into an explicit form, which decrypts how small they may be. For example,
let us pick  $\zeta \in (0,1)$ and confine the choice of $(\mu_\ast,\gamma)$
to the ``semi-strip'' $0< \mu_\ast \leq  \sqrt{v^2+\zeta \Delta_\lambda^2}-v, \gamma >0$ to ensure the first two inequalities in \eqref{limit}. Then $\sigma(\mu_\ast) \geq \Delta_\lambda \sqrt{1-\zeta}$ and so \eqref{limit} holds whenever
$$
\left(3 b_\tau + \frac{b_\varkappa+2b_v + \gamma +b_n \mu_\ast}{\Delta_\lambda \sqrt{1-\zeta}} + \frac{b_\alpha}{\Delta_\lambda^3 (1-\zeta)^{3/2}}\right) \mu_\ast < \Delta_u.
$$
So recommended parameters lie in the portion of the above semi-strip that is cut off by the upper piece of the hyperbola explicitly described by the last relation provided that $<$ is replaced by $=$ in it.
\begin{remark}
\label{rem.notevr}
\rm
The last claim of Theorem~\ref{th.main} implies that the behavior of the field outside the zone \eqref{m} does not matter. Moreover, Theorem~\ref{th.main} remains true even if the field is defined not everywhere outside \eqref{m}, which is of interest for some theoretical fields like $c \|\blr\|^{-1}$.
\end{remark}
\par
The fact that \eqref{limit} holds for all small enough $\delta$ provides a guideline for experimental tuning of the controller.
To analytically tune it, estimates of the field parameters concerned in \eqref{limit} should be known. Various types of knowledge about the field give rise to various specifications of \eqref{limit}. The next section offers relevant samples.

\section{Some Particular Scenarios}
\label{sec.exampl}
\setcounter{equation}{0}
\subsection{Dynamic Radial Field with Time-Invariant Profile}
\label{sub.radial}
Let it be known that the field is radial $D(t,\bldr) = c f(\|\bldr-\bldr^0(t)\|)$ with time-invariant profile $f: (0,\infty) \to \br$.
Then spatial isolines are circles with radii determined by the associated field levels $d_\ast$. This opens the door for analytical computation of many field parameters from Section~\ref{sec.quant} and underlies transformation of \eqref{limit} to a more explicit form. This may be used to acquire fully analytical conditions of isoline tracking for various special field profiles, which will be illustrated in the next two subsections.
\par
Let the profile $f$ be twice continuously differentiable and decay $f^\prime(z) <0$; we initially assume that $f$ is known. Conversely both field ``intensity'' $c>0$ and its moving center $\bldr^0(\cdot)$ are unknown but obey some known bounds:
 \begin{equation}
\label{known}
0< R_{\text{in}}^- \leq  \left\|\bldr_{\text{in}} - \bldr^0 (0)\right\| \leq R_{\text{in}}^+, \quad \|\dot{\bldr}^0(t)\| \leq v_0, \quad \|\ddot{\bldr}^0(t)\| \leq a_0, \quad  0 < c_- \leq c \leq c_+ .
\end{equation}
The objective is to advance to the field isoline $I$ with the given level $d_0$ and to subsequently track $I$.
\par
\par
The control objective is well posed in the face of uncertainties only if the desired field value $d_0$ lies in the field range for any $c \in [c_-,c_+]$, or equivalently
$ d_0/c_\pm \in \text{Im}\, f$.
The following specification of Proposition~\ref{lem.23} shows that irrespective of the control law, the robot should be maneuverable enough to achieve the control objective, and details the required level of maneuvearbility.
\begin{proposition}
\label{prop.radial}
Suppose that the robot is able to remain on the dynamic isoline $D[t,\bldr(t)] \equiv d_0$ if the field uncertainties satisfy \eqref{known}. Then the robot's speed is no less than the maximal feasible speed of the field center $v \geq v_0$, and the following inequality holds for any $v_n \in [0,v_0]$, where $R_-:= \min\{ f^{-1}(d_0/c_+); f^{-1}(d_0/c_-)\}$ and $f^{-1}(\cdot)$ is the inverse function,
\begin{equation}
\label{nec.radial}
 \frac{a_0 }{\sqrt{v^2-v_n^2}}
   + \frac{\left(\sqrt{v^2- v_n^2} + \sqrt{v_0^2-v_n^2}\right)^2}{R_-\sqrt{v^2-v_n^2}}  \leq \ov{u}.
\end{equation}
\end{proposition}
\par
{\bf Proof:} An elementary calculus exercise, along with Lemma~\ref{lem.relation}, shows that
\begin{multline}
\bldn = -\frac{\bldr - \bldr^0}{\|\bldr - \bldr^0\|}, \quad \bldt = - \mathscr{R}_{\frac{\pi}{2}} \bldn , \quad
\lambda = \spr{\dot{\bldr}^0}{N}, \\ \rho = c |f^\prime|,  \quad
 \varkappa = \frac{1}{\|\bldr-\bldr^0\|}, \quad \tau_\rho =0, \quad n_\rho = \frac{f^\pp}{|f^\prime|},
\quad
v_\rho = 0,
\\
\label{n}
 \omega =  - \frac{\spr{\dot{\bldr}^0}{\bldt}}{\|\bldr-\bldr^0\|},
    \quad \alpha = \spr{\ddot{\bldr}^0}{\bldn}  + \frac{\spr{\dot{\bldr}^0}{\bldt}^2}{\|\bldr - \bldr^0\|}.
\end{multline}
By Proposition~\ref{lem.23}, \eqref{speed} and \eqref{accel} hold.
Putting $v_n:= \spr{\dot{\bldr}^0}{\bldn}, R := \|\bldr-\bldr^0\|$ and noting that $\spr{\dot{\bldr}^0}{T} = \pm \sqrt{\|\dot{\bldr}^0\|^2-v_n^2}$, we see that now \eqref{speed} $\Leftrightarrow |v_n| \leq v$ and \eqref{accel} takes the form
\begin{multline*}
- \overline{u} \leq
  - 2 \frac{\sqrt{\|\dot{\bldr}^0\|^2-v_n^2}}{R}  + \frac{\spr{\ddot{\bldr}^0}{\bldn}  + \frac{\|\dot{\bldr}^0\|^2-v_n^2}{R}}{\sqrt{v^2-v_n^2}} + \frac{1}{R} \sqrt{v^2-v_n^2} ,
 \\
  2 \frac{\sqrt{\|\dot{\bldr}^0\|^2-v_n^2}}{R}  + \frac{\spr{\ddot{\bldr}^0}{\bldn}  + \frac{\|\dot{\bldr}^0\|^2-v_n^2}{R}}{\sqrt{v^2-v_n^2}} + \frac{1}{R} \sqrt{v^2-v_n^2} \leq \overline{u}.
\end{multline*}
This holds for any $\dot{\bldr}^0, \ddot{\bldr}^0$ satisfying the bounds from \eqref{known}. Minimizing and maximizing over feasible $\ddot{\bldr}^0$ shows that
\begin{multline*}
- \overline{u} \leq
  - 2 \frac{\sqrt{\|\dot{\bldr}^0\|^2-v_n^2}}{R}  + \frac{-a_0  + \frac{\|\dot{\bldr}^0\|^2-v_n^2}{R}}{\sqrt{v^2-v_n^2}} + \frac{1}{R} \sqrt{v^2-v_n^2} ,
 \\
  2 \frac{\sqrt{\|\dot{\bldr}^0\|^2-v_n^2}}{R}  + \frac{a_0 + \frac{\|\dot{\bldr}^0\|^2-v_n^2}{R}}{\sqrt{v^2-v_n^2}} + \frac{1}{R} \sqrt{v^2-v_n^2} \leq \overline{u}
\end{multline*}
Putting $\dot{\bldr}^0 := v_0 \bldn$ into $|v_n| \leq v$ assures that $v_0 \leq v$. Minimizing and maximizing over feasible $\dot{\bldr}^0$ yields that
$$
  2 \frac{\sqrt{v_0^2-v_n^2}}{R}  + \frac{a_0  }{\sqrt{v^2-v_n^2}}
   \pm \frac{v^2+v_0^2-2v_n^2}{R\sqrt{v^2-v_n^2}}\leq \overline{u}.
$$
Since the expression following $\pm$ is nonnegative, the sign $-$ can be dropped in $\pm$.
It remains to note that as $c$ ranges over the feasible interval from \eqref{known}, the parameter $R$ runs over the interval with the end-points $f^{-1}(d_0/c_-)$ and $f^{-1}(d_0/c_+)$ due to the isoline equation $d_0 = cf(R)$. \epf
\par
The following corollary of Theorem~\ref{th.main} shows
that slight enhancement of the above necessary conditions is enough for the controller \eqref{c.a} to succeed.
Specifically, we in fact sharpen inequality \eqref{nec.radial} from $\leq$ to $<$ and extend it from the isoline on the transient by reducing $R_-$ in \eqref{nec.radial}.
\begin{proposition}
\label{prop.radial1}
Let the robot be faster than the field center $v>v_0$ and initially far enough from it:
\begin{equation}
\label{far.init}
R_{\text{in}}^- > (2v+4 \pi v_0)/\ov{u}.
\end{equation}
Suppose also that for some $\Delta_u >0$, \eqref{nec.radial} holds for all $v_n \in [0,v_0]$ with $\ov{u}$ replaced by $\ov{u}-\Delta_u$
and $R_-$ given by
\begin{gather*}
R_\pm:= \pm  \max\left\{ \pm R^\prime_\pm; \pm f^{-1}\left( d_0/c_-\right); \pm f^{-1}\left( d_0/c_+\right)\right\},
\\
R^\prime_\pm := R_{\text{in}}^\pm \pm (2v+4 \pi v_0)/\ov{u}.
\end{gather*}
Then there exist parameters $\gamma, \delta$ such that whenever \eqref{known} holds, the controller \eqref{c.a} brings the robot to the desired isoline: $d(t) \to d_0$ as $t \to \infty$.
Specifically, this is true if
\begin{equation}
\label{radial.limt}
0< \mu_\ast := \frac{\gamma \delta}{c_- \displaystyle{\min_{R \in [R_-,R_+]}} |f^\prime(R)|}  <  \sqrt{v^2+(v-v_0)^2}-v, \qquad
\left(\frac{\frac{1}{R_-} + \gamma +b_n \mu_\ast}{\sigma(\mu_\ast)} + \frac{a_0 + \frac{v_0^2}{R_-}}{\sigma(\mu_\ast)^3}\right) \mu_\ast < \Delta_u,
\end{equation}
where
$$
\sigma(\mu_\ast) := \sqrt{(v-v_0)^2 - 2 v \mu_\ast - \mu^2_\ast}, \qquad b_n = \displaystyle{\max_{R \in [R_-,R_+]}} \frac{|f^\pp(R)|}{|f^\prime(R)|}.
$$
\end{proposition}
Inequalities \eqref{radial.limt} are feasible and can be satisfied by taking small enough $\mu:=\gamma \delta$.
Full knowledge of the field profile $f(\cdot)$ is not in fact required: it suffices to know estimates $R_-^0 \leq R \leq R_+^0$ of the radius $R$ of the desired isoline and replace $R_-$ and $R_+$ by their lower $R_-:= \min\{R_-^0; R^\prime_-\}$ and upper $R_+:= \max\{R_+^0; R^\prime_+\}$ bounds, respectively. Furthermore, it suffices to know a lower $\min_{R \in [R_-,R_+]} |f^\prime(R)| \geq b_- >0$ and upper $\max_{R \in [R_-,R_+]} |f^\pp(R)| \leq b_+$ estimates and take $\mu_\ast = \mu c_-^{-1} b_-^{-1}$ and $b_n := b_+/b_-$ in Proposition~\ref{prop.radial1}.
\par
{\bf Proof of Proposition~\ref{prop.radial1}:} Let $d^\prime_\pm := c f(R^\prime_\mp), d_- := \min \{d^\prime_-; d_0\}, d_+ := \max \{d^\prime_+; d_0\}$. Then due to \eqref{m},
\begin{equation}
\label{mmm}
\mathcal{M}
\subset
\{(t,\bldr): R_- \leq \| \bldr - \bldr^0(t)\|  \leq R_+ \}.
\end{equation}
 Thanks to \eqref{n}, Assumptions~\ref{ass.grad} and \ref{ass.last} are fulfilled with $b_\rho=c_-^{-1}\big\{\min_{R \in [R_-,R_+]} |f^\prime(R)| \big\}^{-1}$.
As $t$ ranges over $[0,T_2], T_2= 4 \pi/\ov{u}$, the field center $\bldr^0(t)$ remains in the disc $D_0$ of the radius $4 \pi v_0/\ov{u}$ centered at $\bldr^0(0)$. This disk does not contain $\bldr_{\text{in}}$ due to \eqref{far.init} and the first inequality in \eqref{known}. Meanwhile, the gradient $\nabla D(t,\bldr_{\text{in}})$ rotates through an angle that does not exceed $\pi$ thanks to the first equation in \eqref{n}, and so the first requirement of Assumption~\ref{ass.disk} is satisfied with $k=2$. Since $D_\pm \subset \{\bldr: \|\bldr - \bldr_{\text{in}}\| \leq 2 v/\ov{u}\}$, for $t \in [0,T_2], \bldr \in D_\pm$, we have owing to \eqref{known},
\begin{equation*}
\|\bldr - \bldr^0(t)\| \geq \|\bldr_{\text{in}} - \bldr^0(0)\| - \|\bldr-\bldr_{\text{in}}\| - \|\bldr^0(t) - \bldr^0(0)\|
\geq R_{\text{in}}^- -
2 v/\ov{u} - 4 \pi v_0/\ov{u} = R^\prime_-,
\end{equation*}
and similarly $\|\bldr - \bldr^0(t)\| \leq R^\prime_+$. So the second claim of Assumption~\ref{ass.disk} with $k=2$ holds due to \eqref{mmm}.
\par
Thanks to \eqref{n}, inequality \eqref{deltalambda1} is true with $\Delta_\lambda := v-v_0$, whereas \eqref{deltalambda} means that $\overline{u} - \Delta_u$ is no less than
\begin{equation*}
\left| \pm 2 \frac{\spr{\dot{\bldr}^0}{\bldt}}{\|\bldr-\bldr^0\|}  + \frac{\spr{\ddot{\bldr}^0}{\bldn}  + \frac{\spr{\dot{\bldr}^0}{\bldt}^2}{\|\bldr - \bldr^0\|}}{\sqrt{v^2-\spr{\dot{\bldr}^0}{\bldn}^2}} +\frac{1}{\|\bldr-\bldr^0\|} \sqrt{v^2-\spr{\dot{\bldr}^0}{\bldn}^2}\right|
\end{equation*}
Retracing the arguments from the proof of Proposition~\ref{prop.radial} shows that for $(t,\bldr) \in \mathcal{M}$, the last expression does not exceed that in \eqref{nec.radial}. Hence \eqref{deltalambda} holds and Assumption~\ref{ass.upr} is fulfilled.
\par
Because of \eqref{n} and \eqref{mmm}, now the estimates \eqref{estim} hold with
\begin{equation*}
b_\lambda = v_0, \quad b_\tau =0, \quad b_n = \max_{R \in [R_-,R_+]} \frac{|f^\pp(R)|}{|f^\prime(R)|}, \quad b_\varkappa = \frac{1}{R_-}, \quad
b_v = 0, \quad b_\alpha = a_0 + \frac{v_0^2}{R_-}.
\end{equation*}
It remains to note that this transforms
\eqref{limit} into \eqref{radial.limt} and apply Theorem~\ref{th.main}. \epf
\par
The following remark shows that by sacrificing conservatism, the key condition \eqref{nec.radial} can be simplified.
\begin{remark}
\label{rem.cfinal}
Inequality
\eqref{nec.radial} is true whenever
$$
\frac{a_0}{\sqrt{v^2-v_0^2}} + \frac{(v+v_0)^2}{vR_-} \leq \ov{u}.
$$
\end{remark}
This holds since the addends in the left-hand side are upper bounds for the respective addends in \eqref{nec.radial}.
\par
 Now we discuss more particular scenarios, where the tracking conditions can be further specified.
 \begin{figure}[h]
\centering
\scalebox{0.3}{\includegraphics{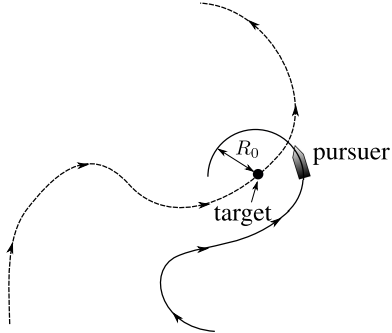}}
\caption{Escorting a maneuvering target.}
\label{fig.pursuit}
\end{figure}

\subsection{Approaching and escorting an unknowingly maneuvering target based on range-only measurements}
\label{subsect.target}
A pursuer robot should approach an unknowingly maneuvering target $\bldr^0(t)$
using only measurements of the relative distance
between them and should subsequently follow the target at the pre-specified range $R_0$, while always maintaining the constant speed $v$; see Fig.~\ref{fig.pursuit}. The bounds from \eqref{known} are still valid; nothing else is known about the target.
\par
This scenario falls under the framework of subsect.~\ref{sub.radial}:
$D[t,\bldr] = - \|\bldr - \bldr^0(t)\|, c=c_\pm =1, f(z)=-z$, and $d_0=-R_0$.
By Proposition~\ref{prop.radial}, the pursuer is capable of carrying out the mission only if its speed is no less than that of the target $v \geq v_0$ and \eqref{nec.radial} holds with $R_-:= R_0$.
Let these minimum requirements be fulfilled in the enhanced form given by Proposition~\ref{prop.radial1}, where $R_-:= \min\{ R_0, R_{\text{in}}^- - (2v+4 \pi v_0)/\ov{u} \}$, and let \eqref{far.init} be true.
By Pro\-po\-si\-tion~\ref{prop.radial1}, the stated objective is ensured by the controller \eqref{c.a} if its parameters satisfy \eqref{radial.limt}, where $\mu_\ast = \gamma \delta$ and $b_n=0$.
\par
In \cite{MTS11ronly}, such problem was solved in the case where more information about the target is available: the target is a kinematically constrained Dubins-like car \eqref{1} with a known range $\ov{u}$ of angular velocities that moves at a constant speed obeying a known upper bound.

\subsection{Detection and tracking of the boundary of a contaminated area transported by  advection}
 \begin{figure}[h]
\centering
\scalebox{0.35}{\includegraphics{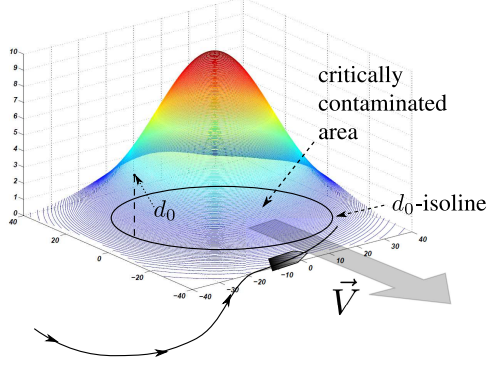}}
\caption{A scalar field transported by advection.}
\label{fig.advection}
\end{figure}
A planar liquid environment is chemically contaminated; see Fig.~\ref{fig.advection}. The concentration $D$ of the pollutant varies over time and plane,
the critically contaminated area $A_{\text{cc}}$ is that where $D$ exceeds a certain level $d_0>0$. A robot should advance to the boundary of $A_{\text{cc}}$ and to subsequently track it, thus displaying $A_{\text{cc}}$. Only the concentration at the robot's current location can be measured.
\par
The initial distribution of the concentration is Gaussian
$
D(0,\bldr) = c e^{-\frac{\|\bldr - \bldr_0\|^2}{2 \sigma^2}}.
$
Its center $\bldr_0$, variation $\sigma^2$, and maximum value $c$ are uncertain, but obey known bounds: $0< \sigma_- \leq \sigma \leq \sigma_+, c \geq c_- > d_0, 0< R_{\text{in}}^- \leq  \left\|\bldr_{\text{in}} - \bldr_0 \right\| \leq R_{\text{in}}^+$.
The liquid moves with a constant velocity $\pmb{V}$ obeying a known bound $\|\pmb{V}\| \leq V_\ast$. Dispersion of the pollutant is negligible and its transport is basically by advection.
\par
Since $D^\prime_t = - \spr{\pmb{V}}{\nabla D}$ \cite{LaLi59} and so $D(t,\bldr) = D(0,\bldr - \pmb{V}t)$,
this scenario falls into the framework of subsect.~\ref{sub.radial} with
$\bldr^0(t) := \bldr_0 + \pmb{V}t, f(z)= e^{-z^2/(2 \sigma^2)}, a_0=0, v_0=V_\ast$.
By Proposition~\ref{prop.radial}, the robot is capable of carrying out the mission only if it is faster than the flow $v \geq V_\ast$ and \eqref{nec.radial} holds with $a_0=0, R_-:= \sigma_-\sqrt{2(\ln c_- - \ln d_0)}$.
Since the first addend in \eqref{nec.radial} is now zero,
\eqref{nec.radial} shapes into $\frac{(v+V_\ast)^2}{R_-} \leq \ov{u} v$ by Remark~\ref{rem.cfinal}.
Here $\ov{u}v$ is the maximal acceleration feasible for the robot \eqref{1}, whereas $\frac{(v+V_\ast)^2}{R_-}$ is the centripetal acceleration required to remain on the desired isoline in the ``worst case'': the flow is maximal and opposes the robot's velocity, the radius of the isoline is minimal over the range of uncertainty in $c$ and $\sigma$.
\par
Let the necessary conditions hold in the enhanced form $v>V_\ast$ and
$(v+V_\ast)^2R_-^{-1} < v\ov{u}$, where $R_-:= \min\{ \sigma_-\sqrt{2(\ln c_- - \ln d_0)}, R_{\text{in}}^- - (2v+4 \pi V_\ast)/\ov{u} \}$, and let \eqref{far.init} hold with $v_0:= V_\ast$.
Modulo elementary computation of the quantities $R_+$ and $b_n$ from \eqref{radial.limt},
Proposition~\ref{prop.radial1} guarantees that the controller \eqref{c.a} succeeds if its parameters satisfy \eqref{radial.limt} with
\begin{gather*}
v_0:=V_\ast, \quad a_0:=0, \quad \Delta_u:= \ov{u}- \frac{(v+V_\ast)^2}{vR_-},
\\
\mu_\ast = \mu \frac{c_-}{\sigma_+^2} \min_{R=R_-,R_+} R e^{- \frac{R^2}{2 \sigma_-^2}} ,\; \text{where} \; \mu = \delta \gamma \; \text{and}\; R_+
\\
:= \max\left\{ \sigma_-\sqrt{2(\ln c_- - \ln d_0)}, R_{\text{in}}^- + (2v+4 \pi V_\ast)/\ov{u} \right\},
\\
b_n = \max_{\substack{R = R_-,R_+\\\sigma=\sigma_-,\sigma_+}} \frac{|\sigma^2 - R^2|}{\sigma^2R}.
\end{gather*}
(The last two lines are justified by elementary estimates of $\min_{z \in [R_-, R_+]} |f^\prime(z)|$ and $\max_{z \in [R_-, R_+]} \frac{|f^\pp(z)|}{|f^\prime(z)|}$.)
\section{A Technical Fact Underlying the Proofs of Proposition~\ref{lem.23} and Theorem~\ref{th.main}}

  \begin{lemma}
The following relations hold:
\begin{gather}
\label{lrho}
\begin{array}{c}
\lambda (t,\bldr+\bldt ds) = \lambda + \omega ds + \so (ds), \\ \lambda (t,\bldr+\bldn ds) = \lambda - v_\rho ds + \so (ds)
\end{array};
\\
\label{frenet-serrat}
\begin{array}{c}
\bldn[t,\bldr+\bldt ds] = \bldn - \kappa \bldt ds +\so(ds), \\ \bldt[t,\bldr+\bldt ds] = \bldt + \kappa \bldn ds +\so(ds);
\end{array}
\\
\label{nnn}
\begin{array}{c}
\bldn(t,\bldr+\bldn ds)= \bldn + \tau_\rho \bldt ds + \so(ds), \\ \bldt(t,\bldr+\bldn ds)= \bldt - \tau_\rho \bldn ds + \so(ds),
\end{array}
\\
\label{n=omega}
\begin{array}{c}
\bldn[t+dt, \bldr_+(dt)] = \bldn - \omega \bldt dt +\so(dt), \\ \bldt[t+dt, \bldr_+(dt)] = \bldt + \omega \bldn dt +\so(dt).
\end{array}
\end{gather}
\end{lemma}
\par
{\bf Proof:}
Formulas \eqref{lrho} are justified by the following:
\begin{multline*}
\lambda(t,\bldr+\bldt ds) \overset{\text{\eqref{speed1}}}{=} -\frac{D^\prime_{t}(t,\bldr+\bldt ds)}{\|\nabla D(t, \bldr+ \bldt ds)\|} = \lambda(t,\bldr)  + \so(ds)
 \\
 - \frac{\spr{\nabla D^\prime_{t}}{\bldt}}{\|\nabla D\|} ds + D^\prime_t \frac{\spr{D^{\prime\prime}\bldn}{\bldt} }{\|\nabla D\|^2} ds
\overset{\text{\eqref{speed1}}}{=} \lambda(t,\bldr) - \frac{\spr{\nabla D^\prime_{t}}{\bldt}}{\|\nabla D\|} ds
\\
- \lambda \frac{\spr{D^{\prime\prime}\bldn}{\bldt} }{\|\nabla D\|} ds + \so(ds)
=
\lambda(t,\bldr) - \frac{\spr{\nabla D^\prime_{t}+ \lambda D^{\prime\prime}\bldn}{\bldt}}{\|\nabla D\|} ds
\\
+ \so(ds)
\overset{\text{\eqref{alpha}}}{=} \lambda + \omega ds + \so (ds); \quad
\lambda(t,\bldr+\bldn ds) = \lambda(t,\bldr)
\\
- \frac{\spr{\nabla D^\prime_{t}+ \lambda D^{\prime\prime}\bldn}{\bldn}}{\|\nabla D\|} ds + \so(ds) \overset{\text{\eqref{alpha}}}{=}
\lambda - v_\rho ds + \so (ds).
\end{multline*}
Formulas \eqref{frenet-serrat} are the Frenet-Serrat equations. Furthermore
\begin{multline*}
N(t,\bldr+\bldn ds) = \frac{\nabla D(t,\bldr+\bldn ds)}{\|\nabla D(t,\bldr+\bldn ds)\|} = \bldn + \frac{D^\pp \bldn}{\|\nabla D\|}ds \\ - \nabla D \frac{\spr{D^\pp \bldn}{\nabla D}}{\|\nabla D\|^3}ds + \so(ds)
 = \bldn + \frac{D^\pp \bldn - \bldn \spr{D^\pp \bldn}{\bldn} }{\|\nabla D\|}ds \\ + \so(ds)
 = \bldn + \frac{\spr{D^\pp \bldn}{\bldt} }{\|\nabla D\|}\bldt ds + \so(ds) \overset{\text{\eqref{tau}}}{=} \bldn + \tau_\rho \bldt ds + \so(ds).
  \end{multline*}
This yields the entire \eqref{nnn} since $\bldn = \mathscr{R}_{\frac{\pi}{2}} \bldt, \bldt = - \mathscr{R}_{\frac{\pi}{2}} \bldn$;  \eqref{n=omega} follows from the definition of $\omega$. \epf

\bibliographystyle{plain}
  \bibliography{Hamidref}
 \end{document}